\documentclass{amsart}
\usepackage[margin=1.00in]{geometry}

\usepackage{amsfonts}

\usepackage{setspace}
\usepackage{graphicx}

	\usepackage{mathrsfs}
\usepackage{hyperref}
\usepackage{graphicx}
\usepackage{amsmath, amsthm, amscd, amsfonts, amssymb, graphicx, color}
 \usepackage{url}
\usepackage{enumerate}
 \makeatletter
\let\reftagform@=\tagform@
\def\tagform@#1{\maketag@@@{(\ignorespaces\textcolor{blue}{#1}\unskip\@@italiccorr)}}
\renewcommand{\eqref}[1]{\textup{\reftagform@{\ref{#1}}}}
\makeatother
\usepackage{hyperref}
\hypersetup{colorlinks=true, linkcolor=red, anchorcolor=green,
citecolor=cyan, urlcolor=red, filecolor=magenta, pdftoolbar=true}

\usepackage{epsfig}        
\usepackage{epic,eepic}       

\newtheorem{theorem}{Theorem}
\theoremstyle{plain}

\newtheorem{corollary}{Corollary}

\newtheorem{lemma}{Lemma}

\newtheorem{remark}{Remark}

\numberwithin{equation}{section}

\begin{document}
\title[ On Cauchy--Schwarz type inequalities]{ On Cauchy--Schwarz type inequalities and applications to numerical radius inequalities}
\author[M.W. Alomari]{Mohammad W. Alomari}

\address {Department of Mathematics, Faculty of Science and Information	Technology, Irbid National University,  P.O. Box 2600, Irbid, P.C. 21110, Jordan.}
\email{mwomath@gmail.com}
\date{\today}


\begin{abstract}

In this work, a refinement of the Cauchy--Schwarz inequality in inner product space is proved.  A more general refinement of the Kato's inequality or the so called mixed Schwarz inequality is established. Refinements of some famous numerical radius inequalities are also pointed out. As shown in this work, these refinements generalize  and refine some recent and old results obtained in literature. Among others, it is proved that if $T\in\mathscr{B}\left(\mathscr{H}\right)$, then
\begin{align*}
\omega^{2}\left(T\right) 
&\le \frac{1}{12} \left\| \left| T \right|+\left| {T^* } \right|\right\|^2 + \frac{1}{3} \omega\left(T\right)\left\| \left| T \right|+\left| {T^* } \right|\right\|  
\\
&\le\frac{1}{6} \left\| \left| T \right|^2+ \left| {T^* } \right|^2 \right\|   + \frac{1}{3} \omega\left(T\right)\left\| \left| T \right|+\left| {T^* } \right|\right\|, 
\end{align*}
which refines the recent inequality obtained by Kittaneh and Moradi in \cite{KM}.
\end{abstract}

\maketitle

 \section{Introduction}
Let $\mathscr{B}\left(\mathscr{H}\right)$ be the Banach algebra
of all bounded linear operators defined on a complex Hilbert space
$\left( \mathscr{H} ;\left\langle \cdot ,\cdot \right\rangle
\right)$  with the identity operator  $1_\mathscr {H}$ in
$\mathscr{B}\left(\mathscr{H}\right)$. For a bounded linear operator $T$ on a Hilbert space
$\mathscr {H}$, the numerical range $W\left(T\right)$ is the image
of the unit sphere of $\mathscr {H}$ under the quadratic form $x\to
\left\langle {Tx,x} \right\rangle$ associated with the operator.
More precisely,
$
W\left( T \right) = \left\{ {\left\langle {Tx,x} \right\rangle :x
	\in \mathscr {H},\left\| x \right\| = 1} \right\}.
$
Also, the numerical radius is defined to be
\begin{align*}
w\left( T \right) = \mathop {\sup }\limits_{\lambda \in W\left( T \right)}  |\lambda|
= \mathop {\sup }\limits_{\left\|
	x \right\| = 1} \left| {\left\langle {Tx,x} \right\rangle }
\right|.
\end{align*}

We recall that,  the usual operator norm of an operator $T$ is defined to be
\begin{align*}
\left\| T \right\|  = \sup \left\{ {\left\| {Tx} \right\|:x \in \mathscr {H},\left\| x \right\| = 1} \right\},
\end{align*}

It's well known that the numerical radius is not submultiplicative, but it  satisfies  $$w(TS)\le 4w\left(T\right) w\left(S\right)$$ for all $T,S\in {\mathscr B}(\mathscr H)$.  In particular if $T,S$ commute,  then $$w(TS)\le 2w\left(T\right) w\left(S\right).$$
Moreover, if $T,S$ are normal  then $w\left(\cdot\right)$ is submultiplicative $w(TS)\le w\left(T\right) w\left(S\right)$.

On the other hand, it is well known that $w\left(\cdot\right)$ defines an operator norm on $\mathscr{B}\left(\mathscr{H}\right)$ which is equivalent to the operator norm $\|\cdot\|$. Moreover, we have
\begin{align}
\frac{1}{2}\|T\|\le w\left(T\right) \le \|T\|\label{eq1.1}
\end{align}
for any $T\in \mathscr{B}\left(\mathscr{H}\right)$.

In 2003, Kittaneh \cite{K1}  refined the right-hand side of \eqref{eq1.1}, by showing that
\begin{align}
w\left(T\right) \le \frac{1}{2}\left\||T|+|T^*|\right\|\le \frac{1}{2}\left(\|T\|+\|T^2\|^{1/2}\right)\label{eq1.2}
\end{align}
for any  $T\in \mathscr{B}\left(\mathscr{H}\right)$.

After that, in \cite{K2}  the same author proved
\begin{align}
\frac{1}{4}\|T^*T+TT^*\|\le  w^2\left(T\right) \le \frac{1}{2}\|T^*T+TT^*\|\label{eq1.3}
\end{align}
for any  $T\in \mathscr{B}\left(\mathscr{H}\right)$. These inequalities are sharp. 

It is well known that, the first inequality in \eqref{eq1.2} is sharper than the second inequality in \eqref{eq1.1}, and the inequalities in \eqref{eq1.3}
refine the inequalities in \eqref{eq1.1}. It is easy to note that, the inequality \eqref{eq1.2} refines the
second inequality in \eqref{eq1.3}.

In \cite{D}, Dragomir  proved the following estimate  of the numerical radius of the product of two Hilbert space operators.
\begin{align}
\omega ^r \left( {S^* T} \right) \le \frac{1}{2}\left\| {\left| T\right|^{2r}  + \left| S \right|^{2r} } \right\|  \qquad (r\ge1).   \label{eq1.4} 
\end{align}

In \cite{EF}, El-Hadad and Kittaneh established two important results that generalize  \eqref{eq1.2} and \eqref{eq1.3}; respectively, as follow:
	\begin{align}
	\omega^{p}\left(T\right) \le \frac{1}{2}\left\| \left| T \right|^{2p\alpha }+\left| {T^* } \right|^{2p\left( {1 - \alpha } \right)}\right\|.\label{eq1.5}
	\end{align}
 and
	\begin{align}
	\label{eq1.6}\omega^{2p}\left(T\right) \le  \left\|\alpha \left| T \right|^{2p }+\left( {1 - \alpha } \right)\left| {T^* } \right|^{2p}\right\|.
	\end{align}
for any $T \in \mathscr{B}\left(\mathscr{H}\right)$, where $0\le \alpha \le 1$  and $p\ge1$.

In his work \cite{Alomari3}, the author of this paper established two refinements of the first inequality in \eqref{eq1.2} and the right-hand side inequality in \eqref{eq1.3}; respectively, by proving that
 \begin{align*}
\omega ^2 \left( T \right) \le \frac{1}{4}\left\| {\left| T \right| + \left| {T^* } \right|} \right\|^2  - \frac{1}{4}\inf \left\langle {\left( {\left| T \right| - \left| {T^* } \right|} \right)x,x} \right\rangle ^2, 
\end{align*}
and
\begin{align*}
\omega ^2 \left( T \right) \le \frac{1}{2}\left\| {T^* T + TT^* } \right\|^2  - \frac{1}{2}\inf \left\langle {\left( {\left| T \right| - \left| {T^* } \right|} \right)x,x} \right\rangle ^2.   
\end{align*}

The Cauchy--Schwarz inequality states that for all vectors $x$ and $y$ in an inner
product space
\begin{align}
\left| {\left\langle {x,y} \right\rangle } \right|    \le     \left\| x \right\|  \left\| y \right\| \label{eq1.7}
\end{align}
where $\left\langle {\cdot,\cdot} \right\rangle$ is the inner product and $\left\|x\right\|=\sqrt{\left\langle {x,x} \right\rangle}$.

Recently,  Kittaneh and Moradi \cite{KM} established the following refinement of \eqref{eq1.7}.
\begin{align}
\left| {\left\langle {x,y} \right\rangle } \right|^2  \le \left| {\left\langle {x,y} \right\rangle } \right|\left\| x \right\|\left\| y \right\| + \frac{1}{2}\left( {\left\| x \right\|^2 \left\| y \right\|^2  - \left| {\left\langle {x,y} \right\rangle } \right|^2 } \right) \le \left\| x \right\|^2 \left\| y \right\|^2. 
\label{eq1.8}
\end{align}
By employing \eqref{eq1.8}, the authors in \cite{KM} established the inequality
\begin{align}
\omega ^2 \left( {S^* T} \right) \le \frac{1}{6}\left\| {\left| T\right|^4  + \left| S \right|^4 } \right\| + \frac{1}{3}\omega \left( {S^*T} \right)\left\| {\left|T \right|^2  + \left| S \right|^2 } \right\|\le \frac{1}{2}  \left\| {\left| T\right|^{4}  + \left| S \right|^{4} } \right\| \label{eq1.9} 
\end{align}
for all Hilbert space operators $T,S\in   {\mathscr{B}}\left(\mathscr H\right) $. It should be noted that, the inequality \eqref{eq1.9} refines \eqref{eq1.4} when $r=2$.

On the other hand, the classical Schwarz inequality for positive operators reads that if $T\in \mathscr{B}\left(\mathscr{H}\right) $ is a positive operator, then
\begin{align}
\left| {\left\langle {Tx,y} \right\rangle} \right|  ^2  \le
\left\langle {T x,x} \right\rangle \left\langle { T y,y}
\right\rangle  \label{eq1.10}
\end{align}
for any   vectors $x,y\in \mathscr{H}$.

In 1952, Kato  \cite{TK} introduced a companion of
Schwarz inequality \eqref{eq1.10}, sometimes known as the Kato's inequality or the so called mixed Schwarz inequality,  which
asserts:
\begin{align}
\label{eq1.11}\left| {\left\langle {Tx,y} \right\rangle} \right|  ^2  \le
\left\langle {\left| T \right|^{2\alpha } x,x} \right\rangle
\left\langle {\left| {T^* } \right|^{2\left( {1 - \alpha }
		\right)} y,y} \right\rangle, \qquad 0\le \alpha \le 1 
\end{align}
for all operators $T\in   {\mathscr{B}}\left(\mathscr H\right) $ and any vectors $x,y\in \mathscr {H}$. In particular, we have
\begin{align}
\left| {\left\langle {Tx,x} \right\rangle} \right|   \le \sqrt{
	\left\langle {\left| T \right| x,x} \right\rangle
	\left\langle {\left| {T^* } \right|  x,x} \right\rangle}.\label{eq1.12}
\end{align}

The following refinement of \eqref{eq1.12} was proved in \cite{KM}:
\begin{align}
\left| {\left\langle {Tx,x} \right\rangle } \right|^2  &\le \frac{1}{3} \left\langle {\left| T \right|x,x} \right\rangle   \left\langle {\left| {T^* } \right|x,x} \right\rangle   + \frac{2}{3}\left| {\left\langle {Tx,x} \right\rangle } \right|\sqrt {\left| {\left\langle {\left| T \right|x,x} \right\rangle } \right|\left| {\left\langle {\left| {T^* } \right|x,x} \right\rangle } \right|}  \label{eq1.13}
\\ 
&\le \left\langle {\left| T \right|x,x} \right\rangle  \left\langle {\left| {T^* } \right|x,x} \right\rangle. \nonumber
\end{align}
As a direct application of \eqref{eq1.13}, the authors of \cite{KM} established the inequality
 \begin{align}
\omega ^2 \left( { T} \right) \le \frac{1}{6}\left\| {\left| T \right|^2  + \left| T^* \right|^2 } \right\| + \frac{1}{3}\omega \left( { T} \right)\left\| {\left| T \right|  + \left| T^* \right| } \right\| \le  \frac{1}{2}\left\| {\left| T \right|^2  + \left| T^* \right|^2 } \right\|\label{eq1.14}
\end{align}
 for any Hilbert space operator $T \in \mathscr{B}\left(\mathscr{H}\right)$, which indeed refines the right-hand side of \eqref{eq1.3}. 
 
 It should mentioned that, another refinement of \eqref{eq1.11} was also presented in the recent work  \cite{Alomari1}. Also,  a Cartesian decomposition version of Kato's type was obtained in \cite{Alomari2}. \\
 
Another direct and simple proof of \eqref{eq1.14}, can be deduced using \eqref{eq1.2} and \eqref{eq1.3}, as obtained in the following steps:
  \begin{align*}
  \omega ^2 \left( { T} \right)  =\frac{1}{3} \omega ^2 \left( { T} \right)+ \frac{2}{3}\omega^2 \left( { T} \right)
  &\le \frac{1}{3} \left(\frac{1}{2}\left\| {\left| T \right|^2  + \left| T^* \right|^2 } \right\| \right)+ \frac{2}{3}\omega  \left( { T} \right)\omega  \left( { T} \right)
   \\
  &\le \frac{1}{3} \left(\frac{1}{2}\left\| {\left| T \right|^2  + \left| T^* \right|^2 } \right\| \right)+ \frac{2}{3}\omega  \left( { T} \right)  \left( { \frac{1}{2}\left\| {\left| T \right|  + \left| T^* \right| } \right\|} \right)
  \\
  &= \frac{1}{6}\left\| {\left| T \right|^2  + \left| T^* \right|^2 } \right\| + \frac{1}{3}\omega \left( { T} \right)\left\| {\left| T \right|  + \left| T^* \right| } \right\|.
 \end{align*}
 Moreover, by applying the first inequality in \eqref{eq1.2} for the second term in the first inequality in \eqref{eq1.14} and then using \eqref{eq2.3}, we get the second inequality in \eqref{eq1.14}. The same approach can be considered to give another proof of \eqref{eq1.9}; by applying the  inequality \eqref{eq1.4} with $r=2$.\\

 In this work, a refinement of the Cauchy--Schwarz inequality in inner product space is proved.  A refinement of the Kato's inequality or the so called mixed Schwarz inequality \eqref{eq1.11} is established. Refinements of some famous numerical radius inequalities are also pointed out. As shown in this work, these refinements are stronger than the recent result \eqref{eq1.9} and \eqref{eq1.14}, and therefore refine the previous ones \eqref{eq1.4}, \eqref{eq1.5}, and \eqref{eq1.6}.

\section{ Refinement of the Cuachy--Shwarz inequaity}
In order to prove our results we need  a sequence of lemmas.
\begin{lemma}\cite{MPF}
	\label{lemma1}   The Power-Mean inequality states that
		\begin{align}
		a^\alpha  b^{1 - \alpha }  \le \alpha a + \left( {1 - \alpha } \right)b \le \left( {\alpha a^p  + \left( {1 - \alpha } \right)b^p } \right)^{\frac{1}{p}}    \label{eq2.1}
		\end{align}
		for all  $\alpha \in \left[0,1\right]$, $a,b\ge0$ and $ p\ge1$.
	 \end{lemma}
\begin{lemma} \cite[Theorem 1.4]{FMPS}
	\label{lemma2}  Let  $T\in \mathscr{B}\left( \mathscr{H}\right)^+ $, then
	\begin{align}
	\left\langle {Tx,x} \right\rangle ^p  \le \left\langle {T^p x,x} \right\rangle, \qquad p\ge1  \label{eq2.2}
	\end{align}
	for any vector $x\in\mathscr{H}$. The inequality \eqref{eq2.2} is reversed if $0\le p\le 1$.
\end{lemma}

\begin{lemma}\cite[Theorem 2.3]{AS}
Let $f$   be a non-negative convex function on $\left[0,\infty\right)$, and let $T,S\in  \mathscr{B}\left( \mathscr{H}\right)$ be two positive operators. Then, 
\begin{align}
\left\|f\left(\frac{T+S}{2}\right)\right\|\le \left\|\frac{f\left(T\right)+f\left(S\right)}{2}\right\|.\label{eq2.3}
\end{align} 
\end{lemma}
 
\begin{lemma}
For all vectors $x$ and $y$ in an inner
product space, we have	
\begin{align}
 \left| {\left\langle {x,y} \right\rangle } \right|^2    \le  \left(1-\beta\right)\left| {\left\langle {x,y} \right\rangle } \right|\left\| x \right\|\left\| y \right\| + \beta \left\| x \right\|^2 \left\| y \right\|^2\le  \left\| x \right\|^2 \left\| y \right\|^2   \label{eq2.4}
\end{align}
for all $\beta\in \left[0,1\right]$.
\end{lemma}
\begin{proof}
Applying the Cauchy--Schwarz inequality \eqref{eq1.7}, we have
\begin{align*}
\left| {\left\langle {x,y} \right\rangle } \right|^2   =\left(1-\beta\right)\left| {\left\langle {x,y} \right\rangle } \right|^2+ \beta \left| {\left\langle {x,y} \right\rangle } \right|^2  
&\le \left(1-\beta\right)\left| {\left\langle {x,y} \right\rangle } \right|\left\| x \right\|\left\| y \right\| + \beta \left\| x \right\|^2 \left\| y \right\|^2   
\\ 
&\le \left(1-\beta\right) \left\| x \right\|^2 \left\| y \right\|^2 + \beta \left\| x \right\|^2 \left\| y \right\|^2
\\ 
&\le  \left\| x \right\|^2 \left\| y \right\|^2, 
\end{align*}
 as required.
 \end{proof}

The following result generalize and refine the Kato's inequality \eqref{eq1.11}, which in turn generalizes \eqref{eq1.13}.
\begin{lemma}
\label{lemma5}	Let $T,S\in \mathscr{B}\left(\mathscr{H}\right)$, $0\le \alpha\le1$ and $p\ge1$. Then	
\begin{align}
\left| {\left\langle {Tx,y} \right\rangle } \right|^{2p}  &\le \beta \left\langle {\left| T \right|^{2p\alpha } x,x} \right\rangle \left\langle {\left| {T^* } \right|^{2p\left( {1 - \alpha } \right)} y,y} \right\rangle  + \left( {1 - \beta } \right)\left| {\left\langle {Tx,y} \right\rangle } \right|^p\sqrt {\left\langle {\left| T \right|^{2p\alpha } x,x} \right\rangle \left\langle {\left| {T^* } \right|^{2p\left( {1 - \alpha } \right)} y,y} \right\rangle }  \label{eq2.5}
\\ 
&\le \left\langle {\left| T \right|^{2p\alpha } x,x} \right\rangle \left\langle {\left| {T^* } \right|^{2p\left( {1 - \alpha } \right)} y,y} \right\rangle. \nonumber
\end{align}
\end{lemma}
\begin{proof}
Using \eqref{eq2.2}, one can easily obtained that	
\begin{align} 
&\beta  \left\langle {\left| T \right|^{2p\alpha } x,x} \right\rangle  \left\langle {\left| {T^* } \right|^{2p\left( {1 - \alpha } \right)} y,y} \right\rangle  + \left( {1 - \beta } \right)\left| {\left\langle {Tx,y} \right\rangle } \right|^p\sqrt {\left\langle {\left| T \right|^{2p\alpha } x,x} \right\rangle \left\langle {\left| {T^* } \right|^{2p\left( {1 - \alpha } \right)} y,y} \right\rangle } 
\nonumber\\
&\ge\beta  \left\langle {\left| T \right|^{2\alpha } x,x} \right\rangle^p \left\langle {\left| {T^* } \right|^{2\left( {1 - \alpha } \right)} y,y} \right\rangle ^p  + \left( {1 - \beta } \right)\left| {\left\langle {Tx,y} \right\rangle } \right|^p\sqrt {\left\langle {\left| T \right|^{2\alpha } x,x} \right\rangle^p \left\langle {\left| {T^* } \right|^{2\left( {1 - \alpha } \right)} y,y} \right\rangle^p } \label{eq2.6} 
\\ 
&= \beta \left| {\left\langle {Tx,y} \right\rangle } \right|^{2p}   + \left( {1 - \beta } \right)\left| {\left\langle {Tx,y} \right\rangle } \right|^{p} \left| {\left\langle {Tx,y} \right\rangle } \right|^{p}  
\nonumber\\ 
&= \left| {\left\langle {Tx,y} \right\rangle } \right|^{2p}   \nonumber
\end{align}
for all $\beta \in \left[0,1\right]$ and $p\ge1$.

On the other hand, we have
\begin{align}
&\beta \left\langle {\left| T \right|^{2p\alpha } x,x} \right\rangle \left\langle {\left| {T^* } \right|^{2p\left( {1 - \alpha } \right)} y,y} \right\rangle  + \left( {1 - \beta } \right)\left| {\left\langle {Tx,y} \right\rangle } \right|^p\sqrt {\left\langle {\left| T \right|^{2p\alpha } x,x} \right\rangle \left\langle {\left| {T^* } \right|^{2p\left( {1 - \alpha } \right)} y,y} \right\rangle }  
\nonumber\\ 
&\le \beta \left\langle {\left| T \right|^{2p\alpha } x,x} \right\rangle  \left\langle {\left| {T^* } \right|^{2p\left( {1 - \alpha } \right)} x,x} \right\rangle  
\nonumber\\
&\qquad\qquad+ \left( {1 - \beta } \right)\sqrt {\left\langle {\left| T \right|^{2p\alpha } x,x} \right\rangle \left\langle {\left| {T^* } \right|^{2p\left( {1 - \alpha } \right)} y,y} \right\rangle } \sqrt {\left\langle {\left| T \right|^{2p\alpha } x,x} \right\rangle \left\langle {\left| {T^* } \right|^{2p\left( {1 - \alpha } \right)} y,y} \right\rangle } \label{eq2.7}
\\ 
&= \beta \left\langle {\left| T \right|^{2p\alpha } x,x} \right\rangle \left\langle {\left| {T^* } \right|^{2p\left( {1 - \alpha } \right)} y,y} \right\rangle  + \left( {1 - \beta } \right)\left\langle {\left| T \right|^{2p\alpha } x,x} \right\rangle \left\langle {\left| {T^* } \right|^{2p\left( {1 - \alpha } \right)} y,y} \right\rangle  
\nonumber\\ 
&= \left\langle {\left| T \right|^{2p\alpha } x,x} \right\rangle \left\langle {\left| {T^* } \right|^{2p\left( {1 - \alpha } \right)} y,y} \right\rangle.\nonumber
\end{align}
Combining \eqref{eq2.6} and \eqref{eq2.7}, we infer that
\begin{align*}
\left| {\left\langle {Tx,y} \right\rangle } \right|^{2p}  &\le \beta \left\langle {\left| T \right|^{2p\alpha } x,x} \right\rangle \left\langle {\left| {T^* } \right|^{2p\left( {1 - \alpha } \right)} y,y} \right\rangle  + \left( {1 - \beta } \right)\left| {\left\langle {Tx,y} \right\rangle } \right|^p\sqrt {\left\langle {\left| T \right|^{2p\alpha } x,x} \right\rangle \left\langle {\left| {T^* } \right|^{2p\left( {1 - \alpha } \right)} y,y} \right\rangle }  
\\ 
&\le \left\langle {\left| T \right|^{2p\alpha } x,x} \right\rangle \left\langle {\left| {T^* } \right|^{2p\left( {1 - \alpha } \right)} y,y} \right\rangle \nonumber
\end{align*}
for any $p\ge1$, which proves \eqref{eq2.5}.
\end{proof}
 
\begin{remark}
Setting $y=x$ in \eqref{eq2.5}, and then choosing $p=1$, $\alpha=\frac{1}{2}$ and $\beta= \frac{1}{3}$, we refer to \eqref{eq1.13}. More generally, for $p=1$ the inequality \eqref{eq2.5} refines the classical Kato's inequality \eqref{eq1.11}.
\end{remark}

\begin{remark}
	Let $T,S\in \mathscr{B}\left(\mathscr{H}\right) $, then by Cauchy--Schwarz inequality \eqref{eq1.7}  we have
	\begin{align*}
	\left| {\left\langle {Tx,Sy} \right\rangle} \right|  \le \left\|Tx\right\|\left\|Sy\right\|  
 =\left\langle {T^*T x,x} \right\rangle^{\frac{1}{2}} \left\langle { S^*S y,y}
	\right\rangle^{\frac{1}{2}} 
 =\left\langle {\left|T\right|^2 x,x} \right\rangle^{\frac{1}{2}} \left\langle { \left|S\right|^2 y,y}
	\right\rangle^{\frac{1}{2}}. 
	\end{align*}
	Hence,
	\begin{align*}
	\left| {\left\langle {Tx,Sy} \right\rangle} \right|^2  \le\left\langle {\left|T\right|^2 x,x} \right\rangle  \left\langle { \left|S\right|^2 y,y}\right\rangle. 
	\end{align*}
	for any   vectors $x,y\in \mathscr{H}$. In particular, for $S=T^*$, we get
	\begin{align*}
	\left| {\left\langle {T^2x,y} \right\rangle} \right|^2  \le\left\langle {\left|T\right|^2 x,x} \right\rangle  \left\langle { \left|T^*\right|^2 y,y}\right\rangle. 
	\end{align*}
	Moreover,   one can esaily observe the following  refinement of \eqref{eq1.10},
	\begin{align*}
	\left| {\left\langle {Tx,Sy} \right\rangle } \right|^{2}  &\le \beta \left\langle {\left| T \right|^{2} x,x} \right\rangle \left\langle {\left| {S } \right|^{2} y,y} \right\rangle  + \left( {1 - \beta } \right)\left| {\left\langle {Tx,Sy} \right\rangle } \right| \sqrt {\left\langle {\left| T \right|^{2 } x,x} \right\rangle \left\langle {\left| {S} \right|^{2} y,y} \right\rangle }  
	\\ 
	&\le \left\|Tx\right\|^2\left\|Sy\right\|^2  
	\end{align*}
	for any   vectors $x,y\in \mathscr{H}$ and all $\beta\in\left[0,1\right]$.
\end{remark}
\section{Applications to numerical radius inequalities}

\subsection{Numerical radius inequalities for a product of two Hilbert space operators}

The following result generalizes the Kittaneh--Moradi inequality \eqref{eq1.9}.
\begin{theorem}
\label{thm1}	Let $T,S\in \mathscr{B}\left(\mathscr{H}\right)$. Then
	\begin{align}
	\omega ^{2r} \left( {S^* T} \right) \le \frac{1}{2} \left(1-\beta\right)  \omega^r \left( {S^*T} \right)\left\| {\left|T \right|^{2r}  + \left| S \right|^{2r} } \right\|  +\frac{1}{2}\beta \left\| {\left| T\right|^{4r}  + \left| S \right|^{4r} } \right\| \le \frac{1}{2} \left\| {\left| T\right|^{4r}  + \left| S \right|^{4r} } \right\| \label{eq3.1}
	\end{align}
	for all $\beta\in\left[0,1\right]$ and $r\ge1$.
\end{theorem}
\begin{proof}
	Let $u\in \mathscr{H}$ be a unit vector. Setting $x=Tu$ and $y=Su$ in the first inequality in \eqref{eq2.4}, we get
	\begin{align*}
	\left| {\left\langle {Tu,Su} \right\rangle } \right|^2 &= \left| {\left\langle {S^*Tu,u} \right\rangle } \right|^2   
	\\
	&\le  \left(1-\beta\right)\left| {\left\langle {Tu,Su} \right\rangle } \right|\left\| Tu \right\|\left\| Su\right\| + \beta \left\| Tu \right\|^2 \left\|S u\right\|^2 
	\\
	&=  \left(1-\beta\right)\left| {\left\langle {S^*Tu,u} \right\rangle } \right|
	\left\langle {\left| T \right|^2 u,u} \right\rangle ^{\frac{1}{2}}  \left\langle {\left| S \right|^2 u,u} \right\rangle ^{\frac{1}{2}} + \beta  \left\langle {\left| T \right|^2 u,u} \right\rangle   \left\langle {\left| S \right|^2 u,u} \right\rangle. 
	\end{align*}
	Empolying the power mean inequality \eqref{eq2.1}, we have
	\begin{align*}
	\left| {\left\langle {S^*Tu,u} \right\rangle } \right|^2  \le \left(   \left(1-\beta\right)\left| {\left\langle {S^*Tu,u} \right\rangle } \right|^r
	\left\langle {\left| T \right|^2 u,u} \right\rangle ^{\frac{r}{2}}  \left\langle {\left| S \right|^2 u,u} \right\rangle ^{\frac{r}{2}} + \beta  \left\langle {\left| T \right|^2 u,u} \right\rangle^r   \left\langle {\left| S \right|^2 u,u} \right\rangle^r\right)^{\frac{1}{r}}, 
	\end{align*}
	which implies that
	\begin{align*}
	&\left| {\left\langle {S^*Tu,u} \right\rangle } \right|^{2r}  
	\\
	&\le     \left(1-\beta\right)\left| {\left\langle {S^*Tu,u} \right\rangle } \right|^r
	\left\langle {\left| T \right|^2 u,u} \right\rangle ^{\frac{r}{2}}  \left\langle {\left| S \right|^2 u,u} \right\rangle ^{\frac{r}{2}} + \beta  \left\langle {\left| T \right|^2 u,u} \right\rangle^r   \left\langle {\left| S \right|^2 u,u} \right\rangle^r 
	\\
	&\le\left(1-\beta\right)\left| {\left\langle {S^*Tu,u} \right\rangle } \right|^r
	\left\langle {\left| T \right|^{2r}u,u} \right\rangle ^{\frac{1}{2}}  \left\langle {\left| S \right|^{2r} u,u} \right\rangle ^{\frac{1}{2}} + \beta  \left\langle {\left| T \right|^{2r} u,u} \right\rangle    \left\langle {\left| S \right|^{2r} u,u} \right\rangle \qquad  \qquad \qquad\text{(by \eqref{eq2.2})}
	\\
	&\le \frac{1}{2} \left(1-\beta\right)\left| {\left\langle {S^*Tu,u} \right\rangle } \right|^r
	\left(\left\langle {\left| T \right|^{2r} u,u} \right\rangle  +  \left\langle {\left| S \right|^{2r} u,u} \right\rangle \right) + \frac{1}{2} \beta \left( \left\langle {\left| T \right|^{4r} u,u} \right\rangle +   \left\langle {\left| S \right|^{4r} u,u} \right\rangle  \right)  \qquad \text{(by \eqref{eq2.1})}
	\\
	&= \frac{1}{2} \left(1-\beta\right)\left| {\left\langle {S^*Tu,u} \right\rangle } \right|^r
	\left\langle {\left( \left| T \right|^{2r}+\left| S \right|^{2r} \right) u,u} \right\rangle    + \frac{1}{2} \beta  \left\langle {\left( \left| T \right|^{4r} +\left| S \right|^{4r} \right)u,u} \right\rangle. 
	\end{align*}
	Taking the supremum over all unit vector $u\in \mathscr{H}$ we get the firt inequality in \eqref{eq3.1}.
	
	To prove the second inequality in \eqref{eq3.1}, simply we have
	\begin{align*}
	\omega ^{2r} \left( {S^* T} \right) &\le \frac{1}{2} \left(1-\beta\right)  \omega^r \left( {S^*T} \right)\left\| {\left|T \right|^{2r}  + \left| S \right|^{2r} } \right\|  +\frac{1}{2}\beta \left\| {\left| T\right|^{4r}  + \left| S \right|^{4r} } \right\|
	\\
	 &\le \frac{1}{2} \left(1-\beta\right)    \left( {\frac{1}{2} \left\| {\left|T \right|^{2r}  + \left| S \right|^{2r} } \right\| } \right)\left\| {\left|T \right|^{2r}  + \left| S \right|^{2r} } \right\|  +\frac{1}{2}\beta \left\| {\left| T\right|^{4r}  + \left| S \right|^{4r} } \right\| \qquad \text{(by \eqref{eq1.4})}
	 \\
	 &=\frac{1}{4} \left(1-\beta\right)   \left\| {\left|T \right|^{2r}  + \left| S \right|^{2r} } \right\|^2  +\frac{1}{2}\beta \left\| {\left| T\right|^{4r}  + \left| S \right|^{4r} } \right\|
	 \\
	 &=\frac{1}{4} \left(1-\beta\right)   \left\| { \left( \frac{\left|T \right|^{2r} +\left|S \right|^{2r}}{2}\right)^2    } \right\| +\frac{1}{2}\beta \left\| {\left| T\right|^{4r}  + \left| S \right|^{4r} } \right\|
	  \\
	 &\le\frac{1}{4} \left(1-\beta\right)   \left\| {\frac{\left(2\left|T \right|^{2r} \right)^2   + \left(2\left|S \right|^{2r} \right)^2 }{2} } \right\| +\frac{1}{2}\beta \left\| {\left| T\right|^{4r}  + \left| S \right|^{4r} } \right\|\qquad  \qquad \qquad\text{(by \eqref{eq2.3})}
	 \\
	 &=\frac{1}{2}  \left\| {\left| T\right|^{4r}  + \left| S \right|^{4r} } \right\|,
	\end{align*}
	which proves the second inequality in \eqref{eq3.1}. 
\end{proof}

\begin{remark}
	Clearly, by choosing $r=1$ and $\beta=\frac{1}{3}$ in \eqref{eq3.1} we refer to the  Kittaneh--Moradi inequality \eqref{eq1.9}.
	
\end{remark}
The following result refines \eqref{eq3.1} (hence, \eqref{eq1.9}) and at the same time refines and generalizes \eqref{eq1.4}. Indeed, the presented result is much better than the mentioned inequalities.
\begin{theorem}
	Let $T,S\in \mathscr{B}\left(\mathscr{H}\right)$,  $r\ge1$ and $ \beta\in \left[0,1\right]$. Then
		\begin{align}
	\omega^{2r}\left(S^* T\right) &\le \frac{1}{4}\beta \left\| {\left| T\right|^{2r}  + \left| S \right|^{2r} } \right\|^2 + \frac{1}{2}\left(1-\beta\right)\omega^{r}\left(T\right)\left\| {\left| T\right|^{2r}  + \left| S \right|^{2r} } \right\|\label{eq3.2}
	\\
	&\le\frac{1}{2}\beta \left\| { \left| T \right|^{4r} +  \left| S \right|^{4r} } \right\|  + \frac{1}{2}\left(1-\beta\right)\omega^{r}\left(T\right)\left\| {\left| T\right|^{2r}  + \left| S \right|^{2r} } \right\|\nonumber
	\\
	&\le \frac{1}{2}  \left\| { \left| T \right|^{4r} +  \left| S \right|^{4r} } \right\|. \nonumber
	\end{align}
\end{theorem}
\begin{proof}
 For all $\beta\in \left[0,1\right]$, we have
	\begin{align*}
	\omega^{2r}\left(S^* T\right)&=\beta \omega^{2r}\left(S^* T\right) + \left(1-\beta\right)\omega^{2r}\left(S^* T\right)
	\\
	&= \beta \omega^{2r}\left(S^* T\right) + \left(1-\beta\right)\omega^{r}\left(S^* T\right)\omega^{r}\left(S^* T\right)
	\\
	&\le \frac{1}{4}\beta \left\| {\left| T\right|^{2r}  + \left| S \right|^{2r} } \right\|^2 + \frac{1}{2}\left(1-\beta\right)\omega^{r}\left(T\right)\left\| {\left| T\right|^{2r}  + \left| S \right|^{2r} } \right\|,
	\end{align*}
 	the last inequality follows from \eqref{eq1.4}, which proves the first inequality in \eqref{eq3.2}. To obtain that the second inequality in \eqref{eq3.2}, we employ \eqref{eq2.3}, to get
	\begin{align*}
	\omega^{2r}\left(T\right) 
	&\le \frac{1}{4}\beta \left\| {\left| T\right|^{2r}  + \left| S \right|^{2r} } \right\|^2 + \frac{1}{2}\left(1-\beta\right)\omega^{r}\left(T\right)\left\| {\left| T\right|^{2r}  + \left| S \right|^{2r} } \right\|  
	\\
	&= \frac{1}{4}\beta \left\| {\left(\frac{2\left| T\right|^{2r}  + 2\left| S \right|^{2r}} {2}\right)^2} \right\|  + \frac{1}{2}\left(1-\beta\right)\omega^{r}\left(T\right)\left\| {\left| T\right|^{2r}  + \left| S \right|^{2r} } \right\| 
	\\
	&\le \frac{1}{4}\beta \left\| {\frac{  \left(2\left| T \right|^{2r}\right)^2  + \left(2\left| S \right|^{2r}\right)^2}{2}} \right\|  + \frac{1}{2}\left(1-\beta\right)\omega^{r}\left(T\right)\left\| {\left| T\right|^{2r}  + \left| S \right|^{2r} } \right\|
	\\
	&= \frac{1}{2}\beta \left\| { \left| T \right|^{4r} +  \left| S \right|^{4r} } \right\|  + \frac{1}{2}\left(1-\beta\right)\omega^{r}\left(T\right)\left\| {\left| T\right|^{2r}  + \left| S \right|^{2r} } \right\|,
	\end{align*}
	which gives the second inequality in \eqref{eq3.2}. The third inequality in \eqref{eq3.2} follows from \eqref{eq3.1}.
\end{proof} 
The following result refines the Kittaneh--Moradi inequality \eqref{eq1.9}.
\begin{corollary}
	Let $T,S\in \mathscr{B}\left(\mathscr{H}\right)$. Then
	\begin{align*}
	\omega^{2}\left(S^* T\right) \le \frac{1}{12}  \left\| {\left| T\right|^{2}  + \left| S \right|^{2} } \right\|^2 + \frac{1}{3}\omega\left(T\right)\left\| {\left| T\right|^{2}  + \left| S \right|^{2} } \right\| 
	&\le\frac{1}{6} \left\| { \left| T \right|^{4} +  \left| S \right|^{4} } \right\|  + \frac{1}{3}\omega\left(T\right)\left\| {\left| T\right|^{2}  + \left| S \right|^{2} } \right\|\nonumber
	\\
	&\le \frac{1}{2}  \left\| { \left| T \right|^{4} +  \left| S \right|^{4} } \right\|. \nonumber
	\end{align*}
\end{corollary}
\begin{proof}
Setting $r=1$ and $\beta=\frac{1}{3}$ in \eqref{eq3.2}.
\end{proof}
\begin{remark}
For all vectors $x$ and $y$ in an inner
product space, we have	
\begin{align*}
\left| {\left\langle {x,y} \right\rangle } \right|     \le  \left(1-\beta\right)\sqrt{\left| {\left\langle {x,y} \right\rangle } \right|\left\| x \right\|\left\| y \right\|} + \beta \left\| x \right\|  \left\| y \right\| \le  \left\| x \right\|  \left\| y \right\|.
\end{align*}
The proof is similar to that one given for \eqref{eq2.4}.\\

Now, let $u\in \mathscr{H}$ be a unit vector. Setting $x=Tu$ and $y=Su$ in the above inequality and proceed as in the proof of Theorem  \ref{thm1}, one can easily observe that 
\begin{align*}
\omega ^{r} \left( {S^* T} \right) \le \frac{1}{\sqrt{2}} \left(1-\beta\right)  \omega^{\frac{r}{2}} \left( {S^*T} \right)\left\| {\left|T \right|^{2r}  + \left| S \right|^{2r} } \right\|^{\frac{1}{2}}  +\frac{1}{2}\beta \left\| {\left| T\right|^{2r}  + \left| S \right|^{2r} } \right\| \le \frac{1}{2} \left\| {\left| T\right|^{2r}  + \left| S \right|^{2r} } \right\|
\end{align*}
for all Hilbert space operators  $T,S\in \mathscr{B}\left(\mathscr{H}\right)$, where $\beta\in\left[0,1\right]$ and $r\ge1$.	Clearly, the previous inequality refines the Dragomir's inequality \eqref{eq1.4}. 
\end{remark}

\subsection{Numerical radius inequalities for Hilbert space operators}

\begin{theorem}
Let $T \in \mathscr{B}\left(\mathscr{H}\right)$. Then
\begin{align}
\omega^{2p}\left(T\right) \le \beta \left\|\alpha\left| T \right|^{2p } +\left( {1 - \alpha } \right)\left| {T^* } \right|^{2p}\right\| + \frac{1}{2}\left( {1 - \beta } \right)\omega^{p}\left(T\right) \left\| \left| T \right|^{2p\alpha }+\left| {T^* } \right|^{2p\left( {1 - \alpha } \right)}\right\|\label{eq3.3}
\end{align}
for all $p\ge1$ and $0\le \alpha,\beta \le 1$.
\end{theorem}
\begin{proof}
Let $x\in\mathscr{H}$ be a unit vector. Setting $y=x$ in \eqref{eq2.5}, it follows that 
\begin{align*}
\left| {\left\langle {Tx,x} \right\rangle } \right|^{2p}  &\le \beta \left\langle {\left| T \right|^{2p\alpha } x,x} \right\rangle \left\langle {\left| {T^* } \right|^{2p\left( {1 - \alpha } \right)} x,x} \right\rangle  
\\
&\qquad+ \left( {1 - \beta } \right)\left| {\left\langle {Tx,x} \right\rangle } \right|^p\sqrt {\left\langle {\left| T \right|^{2p\alpha } x,x} \right\rangle \left\langle {\left| {T^* } \right|^{2p\left( {1 - \alpha } \right)} x,x} \right\rangle }  
\\
&\le \beta \left\langle {\left| T \right|^{2p } x,x} \right\rangle^{\alpha} \left\langle {\left| {T^* } \right|^{2p} x,x} \right\rangle^{\left( {1 - \alpha } \right)} \qquad\qquad\qquad\qquad\qquad\qquad \qquad \text{(by \eqref{eq2.2})}
\\
&\qquad + \left( {1 - \beta } \right)\left| {\left\langle {Tx,x} \right\rangle } \right|^p \cdot  \frac{1}{2}\left(  {\left\langle {\left| T \right|^{2p\alpha } x,x} \right\rangle+ \left\langle {\left| {T^* } \right|^{2p\left( {1 - \alpha } \right)} x,x} \right\rangle }   \right)   \qquad \text{(by \eqref{eq2.1})}
\\
&\le \beta \left[\alpha\left\langle {\left| T \right|^{2p } x,x} \right\rangle +\left( {1 - \alpha } \right)\left\langle {\left| {T^* } \right|^{2p} x,x} \right\rangle  \right] \qquad\qquad\qquad\qquad\qquad\qquad   \text{(by \eqref{eq2.1})}
\\
&\qquad+ \frac{1}{2}\left( {1 - \beta } \right)\left| {\left\langle {Tx,x} \right\rangle } \right|^p    \left\langle {\left(\left| T \right|^{2p\alpha }+\left| {T^* } \right|^{2p\left( {1 - \alpha } \right)} \right)x,x} \right\rangle 
\\
&= \beta  \left\langle {\left(\alpha\left| T \right|^{2p } +\left( {1 - \alpha } \right)\left| {T^* } \right|^{2p}\right)x,x} \right\rangle    
\\
&\qquad + \frac{1}{2}\left( {1 - \beta } \right)\left| {\left\langle {Tx,x} \right\rangle } \right|^p    \left\langle {\left(\left| T \right|^{2p\alpha }+\left| {T^* } \right|^{2p\left( {1 - \alpha } \right)} \right)x,x} \right\rangle. 
\end{align*}
Taking the supremum over all unit vector $x\in \mathscr{H}$ we get the required result.
\end{proof}
\begin{remark}
Setting $\beta=\frac{1}{3}$, $\alpha=\frac{1}{2}$ and $p=1$ in \eqref{eq3.3} we get the first inequality in \eqref{eq1.14}. 
\end{remark}

The following result is more stronger than El-Hadad--Kittaneh inequality \eqref{eq1.6}.
\begin{theorem}
	Let $T \in \mathscr{B}\left(\mathscr{H}\right)$. Then
	\begin{align}
	\label{eq3.4}\omega^{2p}\left(T\right) &\le \beta \left\|\alpha\left| T \right|^{2p } +\left( {1 - \alpha } \right)\left| {T^* } \right|^{2p}\right\| +  \left( {1 - \beta } \right)\omega^{p}\left(T\right) \sqrt{\left\|\alpha\left| T \right|^{2p } +\left( {1 - \alpha } \right)\left| {T^* } \right|^{2p}\right\|}
	\\
	&\le  \left\|\alpha\left| T \right|^{2p } +\left( {1 - \alpha } \right)\left| {T^* } \right|^{2p}\right\|   \nonumber 
	\end{align}
	for all $p\ge1$ and $0\le \alpha,\beta \le 1$.
\end{theorem}
\begin{proof}
	Let $x\in\mathscr{H}$ be a unit vector. Setting $y=x$ in  \eqref{eq2.5}, it follows that 
	\begin{align*}
	\left| {\left\langle {Tx,x} \right\rangle } \right|^{2p}  &\le \beta \left\langle {\left| T \right|^{2p\alpha } x,x} \right\rangle \left\langle {\left| {T^* } \right|^{2p\left( {1 - \alpha } \right)} x,x} \right\rangle \nonumber 
	\\
	&\qquad+ \left( {1 - \beta } \right)\left| {\left\langle {Tx,x} \right\rangle } \right|^p\sqrt {\left\langle {\left| T \right|^{2p\alpha } x,x} \right\rangle \left\langle {\left| {T^* } \right|^{2p\left( {1 - \alpha } \right)} x,x} \right\rangle }  \nonumber
	\\
	&\le \beta \left\langle {\left| T \right|^{2p } x,x} \right\rangle^{\alpha} \left\langle {\left| {T^* } \right|^{2p} x,x} \right\rangle^{\left( {1 - \alpha } \right)} \nonumber \qquad\qquad\qquad\qquad\qquad\qquad\qquad  \text{(by \eqref{eq2.2})}
	\\
	&\qquad + \left( {1 - \beta } \right)\left| {\left\langle {Tx,x} \right\rangle } \right|^p\sqrt {\left\langle {\left| T \right|^{2p } x,x} \right\rangle^{\alpha} \left\langle {\left| {T^* } \right|^{2p} x,x} \right\rangle^{\left( {1 - \alpha } \right)} } \nonumber 
	\\
	&\le  \beta  \left\langle {\left(\alpha\left| T \right|^{2p } +\left( {1 - \alpha } \right)\left| {T^* } \right|^{2p}\right)x,x} \right\rangle  
	\\
	&\qquad +  \left( {1 - \beta } \right)\left| {\left\langle {Tx,x} \right\rangle } \right|^p   \sqrt{\left\langle {\left(\alpha\left| T \right|^{2p } +\left( {1 - \alpha } \right)\left| {T^* } \right|^{2p}\right)x,x} \right\rangle }. \nonumber\qquad\qquad   \text{(by \eqref{eq2.1})}
	\end{align*}
	Taking the supremum over all unit vector $x\in \mathscr{H}$, we get the first in \eqref{eq3.4}.
	
	To obtain the second inequality from the first inequality we have
	\begin{align*}
	  \omega^{2p}\left(T\right) &\le \beta \left\|\alpha\left| T \right|^{2p } +\left( {1 - \alpha } \right)\left| {T^* } \right|^{2p}\right\| +  \left( {1 - \beta } \right)\omega^{p}\left(T\right) \sqrt{\left\|\alpha\left| T \right|^{2p } +\left( {1 - \alpha } \right)\left| {T^* } \right|^{2p}\right\|}
	\\
	&\le \left\|\alpha\left| T \right|^{2p } +\left( {1 - \alpha } \right)\left| {T^* } \right|^{2p}\right\|   \qquad (\text{by \eqref{eq1.6}})
	\end{align*}
	which proves the required result. 
\end{proof}

\begin{theorem}
	Let $T \in \mathscr{B}\left(\mathscr{H}\right)$. Then
	\begin{align}
	\omega^{p}\left(T\right) \le \frac{1}{2}\beta \left\|\left| T \right|^{2p\alpha } +\left| {T^* } \right|^{2p\left( {1 - \alpha } \right)}\right\| + \frac{1}{\sqrt{2}}\left( {1 - \beta } \right)\omega^{\frac{p}{2}}\left(T\right) \left\| \left| T \right|^{2p\alpha }+\left| {T^* } \right|^{2p\left( {1 - \alpha } \right)}\right\|^{\frac{1}{2}}\label{eq3.5}
	\end{align}
	for all $p\ge1$ and $0\le \alpha,\beta \le 1$.
\end{theorem}
\begin{proof}
	
	Following similar steps as obtained in \eqref{eq2.6} and \eqref{eq2.7}, one can observe that
	\begin{align}
	\left| {\left\langle {Tx,y} \right\rangle } \right|^{p}  &\le \beta \left\langle {\left| T \right|^{2p\alpha } x,x} \right\rangle^{\frac{1}{2}} \left\langle {\left| {T^* } \right|^{2p\left( {1 - \alpha } \right)} y,y} \right\rangle^{\frac{1}{2}} \label{eq3.6}
	\\
	&\qquad\qquad + \left( {1 - \beta } \right)\left| {\left\langle {Tx,y} \right\rangle } \right|^{\frac{p}{2}}\sqrt {\left\langle {\left| T \right|^{2p\alpha } x,x} \right\rangle^{\frac{1}{2}} \left\langle {\left| {T^* } \right|^{2p\left( {1 - \alpha } \right)} y,y} \right\rangle^{\frac{1}{2}} }  
	\nonumber\\ 
	&\le \left\langle {\left| T \right|^{2p\alpha } x,x} \right\rangle^{\frac{1}{2}} \left\langle {\left| {T^* } \right|^{2p\left( {1 - \alpha } \right)} y,y} \right\rangle^{\frac{1}{2}} \nonumber
	\end{align}
	for all $p\ge1$ and $0\le \alpha,\beta\le 1$.
	
	Let $x\in\mathscr{H}$ be a unit vector. Setting $y=x$ in \eqref{eq3.6}, it follows that 
	\begin{align*}
	\left| {\left\langle {Tx,y} \right\rangle } \right|^{p}  &\le \beta \left\langle {\left| T \right|^{2p\alpha } x,x} \right\rangle^{\frac{1}{2}} \left\langle {\left| {T^* } \right|^{2p\left( {1 - \alpha } \right)} x,x} \right\rangle^{\frac{1}{2}} 
	\\
	&\qquad\qquad + \left( {1 - \beta } \right)\left| {\left\langle {Tx,x} \right\rangle } \right|^{\frac{p}{2}}\sqrt {\left\langle {\left| T \right|^{2p\alpha } x,x} \right\rangle^{\frac{1}{2}} \left\langle {\left| {T^* } \right|^{2p\left( {1 - \alpha } \right)} x,x} \right\rangle^{\frac{1}{2}} } 
	\\
	&\le \frac{1}{2}\beta \left(\left\langle {\left| T \right|^{2p\alpha } x,x} \right\rangle + \left\langle {\left| {T^* } \right|^{2p\left( {1 - \alpha } \right)} x,x} \right\rangle \right)   
	\\
	&\qquad\qquad +\frac{1}{\sqrt{2}} \left( {1 - \beta } \right)\left| {\left\langle {Tx,x} \right\rangle } \right|^{\frac{p}{2}}\sqrt {\left(\left\langle {\left| T \right|^{2p\alpha } x,x} \right\rangle + \left\langle {\left| {T^* } \right|^{2p\left( {1 - \alpha } \right)} x,x} \right\rangle \right) }\qquad\text{(by \eqref{eq2.1})}
	\\
	&\le \frac{1}{2}\beta \left\langle {\left(\left| T \right|^{2p\alpha }+\left| {T^* } \right|^{2p\left( {1 - \alpha } \right)} \right)x,x} \right\rangle  
	\\
	&\qquad+ \frac{1}{\sqrt{2}}\left( {1 - \beta } \right)\left| {\left\langle {Tx,x} \right\rangle } \right|^{\frac{p}{2}}    \sqrt{\left\langle {\left(\left| T \right|^{2p\alpha }+\left| {T^* } \right|^{2p\left( {1 - \alpha } \right)} \right)x,x} \right\rangle}. 
	\end{align*}
	Taking the supremum over all unit vector $x\in \mathscr{H}$ we get the required result.
\end{proof}

The following corollary shows that our result \eqref{eq3.5} is more stronger than El-Hadad--Kittaneh inequality \eqref{eq1.5}.
\begin{corollary}
	Let $T \in \mathscr{B}\left(\mathscr{H}\right)$. Then	
	\begin{align*}
	\omega^{p}\left(T\right)&\le  \frac{1}{2}\beta \left\|\left| T \right|^{2p\alpha } +\left| {T^* } \right|^{2p\left( {1 - \alpha } \right)}\right\| + \frac{1}{\sqrt{2}}\left( {1 - \beta } \right)\omega^{\frac{p}{2}} \left(T\right) \left\| \left| T \right|^{2p\alpha }+\left| {T^* } \right|^{2p\left( {1 - \alpha } \right)}\right\|^{\frac{1}{2}} 
	\\
	&\le \frac{1}{2}  \left\|\left| T \right|^{2p\alpha } +\left| {T^* } \right|^{2p\left( {1 - \alpha } \right)}\right\|  
	\end{align*}
	for all $p\ge1$ and $0\le \alpha,\beta \le 1$.	
\end{corollary}
\begin{proof}
	From \eqref{eq3.5}, we have
	\begin{align*}
	\omega^{p}\left(T\right)&\le  \frac{1}{2}\beta \left\|\left| T \right|^{2p\alpha } +\left| {T^* } \right|^{2p\left( {1 - \alpha } \right)}\right\| + \frac{1}{\sqrt{2}}\left( {1 - \beta } \right)\omega^{\frac{p}{2}} \left(T\right) \left\| \left| T \right|^{2p\alpha }+\left| {T^* } \right|^{2p\left( {1 - \alpha } \right)}\right\|^{\frac{1}{2}} 
	\\
	&\le \frac{1}{2}\beta \left\|\left| T \right|^{2p\alpha } +\left| {T^* } \right|^{2p\left( {1 - \alpha } \right)}\right\|+ \frac{1}{2}\left( {1 - \beta } \right) \left\| \left| T \right|^{2p\alpha }+\left| {T^* } \right|^{2p\left( {1 - \alpha } \right)}\right\|   \qquad\text{(by \eqref{eq1.5})}
	 \\
	&= \frac{1}{2}  \left\|\left| T \right|^{2p\alpha } +\left| {T^* } \right|^{2p\left( {1 - \alpha } \right)}\right\|, 
	\end{align*}
	as required.
\end{proof}

The following two results generalize the Kittaneh--Moradi inequality \eqref{eq1.9}.
\begin{theorem}
	Let $T \in \mathscr{B}\left(\mathscr{H}\right)$. Then
	\begin{align}
	\omega^{2p}\left(T\right) \le \frac{1}{2}\beta \left\|\left| T \right|^{4p\alpha } +\left| {T^* } \right|^{4p\left( {1 - \alpha } \right)}\right\| + \frac{1}{2}\left( {1 - \beta } \right)\omega^{p}\left(T\right) \left\| \left| T \right|^{2p\alpha }+\left| {T^* } \right|^{2p\left( {1 - \alpha } \right)}\right\|\label{eq3.7}
	\end{align}
	for all $p\ge1$ and $0\le \alpha,\beta \le 1$.
\end{theorem}
\begin{proof}
	Let $x\in\mathscr{H}$ be a unit vector. Setting $y=x$ in \eqref{eq2.5}, it follows that 
	\begin{align*}
	\left| {\left\langle {Tx,x} \right\rangle } \right|^{2p}  &\le \beta \left\langle {\left| T \right|^{2p\alpha } x,x} \right\rangle \left\langle {\left| {T^* } \right|^{2p\left( {1 - \alpha } \right)} x,x} \right\rangle  
	\\
	&\qquad+ \left( {1 - \beta } \right)\left| {\left\langle {Tx,x} \right\rangle } \right|^p\sqrt {\left\langle {\left| T \right|^{2p\alpha } x,x} \right\rangle \left\langle {\left| {T^* } \right|^{2p\left( {1 - \alpha } \right)} x,x} \right\rangle }  
	\\
	&\le \frac{1}{2}\beta \left(\left\langle {\left| T \right|^{2p\alpha } x,x} \right\rangle^2+ \left\langle {\left| {T^* } \right|^{2p\left( {1 - \alpha } \right)} x,x} \right\rangle^2\right) \qquad \qquad \qquad  \qquad \text{(by \eqref{eq2.1})}
	\\
	&\qquad+ \frac{1}{2}\left( {1 - \beta } \right)\left| {\left\langle {Tx,x} \right\rangle } \right|^p    \left\langle {\left(\left| T \right|^{2p\alpha }+\left| {T^* } \right|^{2p\left( {1 - \alpha } \right)} \right)x,x} \right\rangle \qquad \qquad \text{(by \eqref{eq2.1})}
	\\
	&\le \frac{1}{2}\beta  \left\langle {\left| T \right|^{4p \alpha} x,x} \right\rangle +\left\langle {\left| {T^* } \right|^{4p\left( {1 - \alpha } \right)} x,x} \right\rangle  \qquad \qquad \qquad \qquad\qquad \qquad \text{(by \eqref{eq2.2})}
	\\
	&\qquad+ \frac{1}{2}\left( {1 - \beta } \right)\left| {\left\langle {Tx,x} \right\rangle } \right|^p    \left\langle {\left(\left| T \right|^{2p\alpha }+\left| {T^* } \right|^{2p\left( {1 - \alpha } \right)} \right)x,x} \right\rangle 
	\\
	&= \frac{1}{2}\beta  \left\langle {\left(\left| T \right|^{4p\alpha } +\left| {T^* } \right|^{4p\left( {1 - \alpha } \right)}\right)x,x} \right\rangle    
	\\
	&\qquad + \frac{1}{2}\left( {1 - \beta } \right)\left| {\left\langle {Tx,x} \right\rangle } \right|^p    \left\langle {\left(\left| T \right|^{2p\alpha }+\left| {T^* } \right|^{2p\left( {1 - \alpha } \right)} \right)x,x} \right\rangle 
	\end{align*}
	Taking the supremum over all unit vector $x\in \mathscr{H}$ we get the required result.
\end{proof}

\begin{remark}
Setting $\beta=\frac{1}{3}$, $\alpha=\frac{1}{2}$ and $p=1$ \eqref{eq3.7} we get the Kittaneh--Moradi \eqref{eq1.14}. 
 \end{remark}

\begin{corollary}
	Let $T \in \mathscr{B}\left(\mathscr{H}\right)$. Then	
\begin{align}
\label{eq3.8}	\omega^{2p}\left(T\right)&\le  \frac{1}{2}\beta \left\|\left| T \right|^{4p\alpha } +\left| {T^* } \right|^{4p\left( {1 - \alpha } \right)}\right\| + \frac{1}{2}\left( {1 - \beta } \right)\omega^{p}\left(T\right) \left\| \left| T \right|^{2p\alpha }+\left| {T^* } \right|^{2p\left( {1 - \alpha } \right)}\right\|
\\
&\le \frac{1}{2}  \left\|\left| T \right|^{4p\alpha } +\left| {T^* } \right|^{4p\left( {1 - \alpha } \right)}\right\|  \nonumber
\end{align}
for all $p\ge1$ and $0\le \alpha,\beta \le 1$.	
\end{corollary}
\begin{proof}
	From \eqref{eq3.7}, we have
	\begin{align*}
	\omega^{2p}\left(T\right)&\le  \frac{1}{2}\beta \left\|\left| T \right|^{4p\alpha } +\left| {T^* } \right|^{4p\left( {1 - \alpha } \right)}\right\| + \frac{1}{2}\left( {1 - \beta } \right)\omega^{p}\left(T\right) \left\| \left| T \right|^{2p\alpha }+\left| {T^* } \right|^{2p\left( {1 - \alpha } \right)}\right\|
	\\
	&\le  \frac{1}{2}\beta \left\|\left| T \right|^{4p\alpha } +\left| {T^* } \right|^{4p\left( {1 - \alpha } \right)}\right\|+ \frac{1}{4}\left( {1 - \beta } \right) \left\| \left| T \right|^{2p\alpha }+\left| {T^* } \right|^{2p\left( {1 - \alpha } \right)}\right\|^2 \qquad\qquad \text{(by \eqref{eq1.5})}
	 \\
	&= \frac{1}{2}\beta \left\|\left| T \right|^{4p\alpha } +\left| {T^* } \right|^{4p\left( {1 - \alpha } \right)}\right\|+ \frac{1}{4}\left( {1 - \beta } \right)   \left\|\left( \frac{2\left| T \right|^{2p\alpha }+2\left| {T^* } \right|^{2p\left( {1 - \alpha } \right)}} {2}\right)^2\right\|
	\\
	&\le \frac{1}{2}\beta \left\|\left| T \right|^{4p\alpha } +\left| {T^* } \right|^{4p\left( {1 - \alpha } \right)}\right\| + \frac{1}{8}\left( {1 - \beta } \right)   \left\|\left(  2\left| T \right|^{2p\alpha }\right)^2+\left(2\left| {T^* } \right|^{2p\left( {1 - \alpha } \right)} \right)^2\right\|  \qquad  \text{(by \eqref{eq2.3})}
	\\
	&= \frac{1}{2}\beta \left\|\left| T \right|^{4p\alpha } +\left| {T^* } \right|^{4p\left( {1 - \alpha } \right)}\right\|  + \frac{1}{2}\left( {1 - \beta } \right)   \left\| \left| T \right|^{4p\alpha } + \left| {T^* } \right|^{4p\left( {1 - \alpha } \right)}   \right\| 
	\\
	&= \frac{1}{2}  \left\|\left| T \right|^{4p\alpha } +\left| {T^* } \right|^{4p\left( {1 - \alpha } \right)}\right\|,  
	\end{align*}
	as required.
\end{proof}

We finish this work by noting that, for any $T\in \mathscr{B}\left(\mathscr{H}\right)$, $p\ge1$, and $\alpha,\beta\in \left[0,1\right]$. We use the inequalities \eqref{eq1.5} and \eqref{eq1.6}, to deduce the following refinements of \eqref{eq1.2}, \eqref{eq1.3} and \eqref{eq1.14}.
\begin{enumerate}
\item In fact, we have
\begin{align*}
\omega^{2p}\left(T\right)&=\beta \omega^{2p}\left(T\right) + \left(1-\beta\right)\omega^{2p}\left(T\right)
\\
&= \beta \omega^{2p}\left(T\right) + \left(1-\beta\right)\omega^{p}\left(T\right)\omega^{p}\left(T\right)
\\
&\le \frac{1}{4}\beta \left\| \left| T \right|^{2p\alpha }+\left| {T^* } \right|^{2p\left( {1 - \alpha } \right)}\right\|^2 + \frac{1}{2}\left(1-\beta\right)\omega^{p}\left(T\right)\left\| \left| T \right|^{2p\alpha }+\left| {T^* } \right|^{2p\left( {1 - \alpha } \right)}\right\|,
\end{align*}
which of course refines \eqref{eq3.7}, after applying \eqref{eq2.3}.\\

In particular, for $p=1$, $\alpha=\frac{1}{2}$ and $\beta = \frac{1}{3}$, we get
\begin{align}
\omega^{2}\left(T\right) 
\le \frac{1}{12} \left\| \left| T \right|+\left| {T^* } \right|\right\|^2 + \frac{1}{3} \omega\left(T\right)\left\| \left| T \right|+\left| {T^* } \right|\right\|.  \label{eq3.9}
\end{align}
Moreover, employing \eqref{eq1.2} on \eqref{eq3.9}, we get
\begin{align*}
\omega^{2}\left(T\right) 
&\le \frac{1}{12} \left\| \left| T \right|+\left| {T^* } \right|\right\|^2 + \frac{1}{3} \omega\left(T\right)\left\| \left| T \right|+\left| {T^* } \right|\right\|  
\\
&\le \frac{1}{12} \left\| \left| T \right|+\left| {T^* } \right|\right\|^2 + \frac{1}{3}  \left(\frac{1}{2}\left\| \left| T \right|+\left| {T^* } \right|\right\| \right)\left\| \left| T \right|+\left| {T^* } \right|\right\| 
\\
&= \frac{1}{12} \left\| \left| T \right|+\left| {T^* } \right|\right\|^2 +\frac{2}{12} \left\| \left| T \right|+\left| {T^* } \right|\right\|^2 
\\
&= \frac{1}{4} \left\| \left| T \right|+\left| {T^* } \right|\right\|^2,
\end{align*}
which indeed refines \eqref{eq1.2}. Moreover, \eqref{eq3.9} is much better than the first (the Kittaneh--Moradi) inequality \eqref{eq1.14}. To see that, from \eqref{eq3.9} we have 
\begin{align*}
\omega^{2}\left(T\right) 
&\le \frac{1}{12} \left\| \left| T \right|+\left| {T^* } \right|\right\|^2 + \frac{1}{3} \omega\left(T\right)\left\| \left| T \right|+\left| {T^* } \right|\right\|  
\\
&= \frac{1}{12} \left\| \left(\frac{2\left| T \right|+2\left| {T^* } \right|}{2}\right)^2\right\|  + \frac{1}{3} \omega\left(T\right)\left\| \left| T \right|+\left| {T^* } \right|\right\|   
\\
&\le \frac{1}{24} \left\|\left(2\left| T \right| \right)^2+ \left( 2\left| {T^* } \right|\right)^2\right\|   + \frac{1}{3} \omega\left(T\right)\left\| \left| T \right|+\left| {T^* } \right|\right\| \qquad\qquad (\text{by  \eqref{eq2.3}})
\\
&=\frac{1}{6} \left\| \left| T \right|^2+ \left| {T^* } \right|^2 \right\|   + \frac{1}{3} \omega\left(T\right)\left\| \left| T \right|+\left| {T^* } \right|\right\|, 
\end{align*}
which exactly the first inequality in \eqref{eq1.14}, as required.\\

\item Using similar arguments, from \eqref{eq1.5} and \eqref{eq1.6}, we can deduce the new inequality
\begin{align*}
\omega^{2p}\left(T\right)&=\beta \omega^{2p}\left(T\right) + \left(1-\beta\right)\omega^{2p}\left(T\right)
\\
&= \beta \omega^{2p}\left(T\right) + \left(1-\beta\right)\omega^{p}\left(T\right)\omega^{p}\left(T\right)
\\
&\le \frac{1}{4}\beta \left\| \left| T \right|^{2p\alpha }+\left| {T^* } \right|^{2p\left( {1 - \alpha } \right)}\right\|^2 +  \left(1-\beta\right)\omega^{p}\left(T\right)\left\|\alpha \left| T \right|^{2p }+\left( {1 - \alpha } \right)\left| {T^* } \right|^{2p}\right\|^{\frac{1}{2}}.
\end{align*}
In particular, for $p=1$, $\alpha=\frac{1}{2}$ and $\beta = \frac{1}{3}$, we get
\begin{align}
\omega^{2}\left(T\right) 
&\le \frac{1}{12}  \left\| \left| T \right| +\left| {T^* } \right| \right\|^2 + \frac{\sqrt{2}}{3} \omega \left(T\right)\left\|  \left| T \right|^{2  }+ \left| {T^* } \right|^{2}\right\|^{\frac{1}{2}}   \label{eq3.10}
\\
&\le \frac{1}{12} \left\| \left| T \right|+\left| {T^* } \right|\right\|^2 + \frac{\sqrt{2}}{3 } \left(\frac{1}{\sqrt{2}}\left\| \left| T \right|^2+\left| {T^* } \right|^2\right\|^{\frac{1}{2}} \right)\left\|  \left| T \right|^{2  }+ \left| {T^* } \right|^{2}\right\|^{\frac{1}{2}}  \qquad\qquad (\text{by  \eqref{eq1.3}})\nonumber
\\
&=\frac{1}{12} \left\| \left| T \right|+\left| {T^* } \right|\right\|^2 +   \frac{1}{3} \left\|  \left| T \right|^{2  }+ \left| {T^* } \right|^{2}\right\|.\nonumber 
\end{align}
Moreover,  we have
\begin{align*}
\omega^{2}\left(T\right) 
&\le \frac{1}{12} \left\| \left| T \right|+\left| {T^* } \right|\right\|^2 +   \frac{1}{3} \left\|  \left| T \right|^{2  }+ \left| {T^* } \right|^{2}\right\| 
\\
&= \frac{1}{12} \left\| \left(\frac{2\left| T \right|+2\left| {T^* } \right|}{2}\right)^2\right\| +\frac{2}{6} \left\|  \left| T \right|^{2  }+ \left| {T^* } \right|^{2}\right\| 
\\
&\le\frac{1}{24} \left\|\left(2\left| T \right| \right)^2+ \left( 2\left| {T^* } \right|\right)^2\right\| +\frac{2}{6} \left\|  \left| T \right|^{2  }+ \left| {T^* } \right|^{2}\right\|    \qquad\qquad\qquad    (\text{by  \eqref{eq2.3}})
\\
&= \frac{1}{6} \left\|  \left| T \right|^2+ \left| {T^* } \right|^2 \right\| +\frac{2}{6} \left\|  \left| T \right|^{2  }+ \left| {T^* } \right|^{2}\right\| 
\\
&= \frac{1}{2} \left\| \left| T \right|^2+\left| {T^* } \right|^2\right\|,
\end{align*}
which means that the double inequality in \eqref{eq3.10} present two new stronger refinements than the right-hand side of \eqref{eq1.3}.  

Hence, the inequality \eqref{eq3.9} is stronger than both \eqref{eq1.2} and its refinement \eqref{eq1.14}, as well as \eqref{eq3.10} is much better than the inequality \eqref{eq1.3}. Moreover, since
\begin{align*}
\left\| \left| T \right|+\left| {T^* } \right|\right\|^2 =\left\| \left(\frac{2\left| T \right|+2\left| {T^* } \right|}{2}\right)^2\right\|
&\le\frac{1}{2} \left\|\left(2\left| T \right| \right)^2+ \left( 2\left| {T^* } \right|\right)^2\right\|\qquad  (\text{by  \eqref{eq2.3}})
\\
&=2\left\| \left| T \right|^2+ \left| {T^* } \right|^2\right\|,
\end{align*}
which implies that
\begin{align*}
\left\| \left| T \right|+\left| {T^* } \right|\right\|  \le \sqrt{2}\left\| \left| T \right|^2+ \left| {T^* } \right|^2\right\|^{\frac{1}{2}}.
\end{align*}
Thus, \eqref{eq3.9} refines \eqref{eq3.10}.

\end{enumerate} 

\begin{remark}
The celebrated  Buzano inequality states that \cite{Buzano}: 
\begin{align*}
\left| {\left\langle {x,e} \right\rangle \left\langle {e,y}
	\right\rangle } \right| \le \frac{1}{2}\left( {\left|
	{\left\langle {x,y} \right\rangle } \right| + \left\| x
	\right\|\left\| y \right\|} \right) 
\end{align*}
for every vectors $x,y,e\in \mathscr{H}$ with $\|e\|=1$. It should mentioned that, by employing the inequality \eqref{eq2.4} and/or use the same considered techniques, simple computations could give more than one refinements of the Buzano inequality. Thus, several numerical radius inequalities could be stated. We leave the details to the interested reader. 
\end{remark}

\begin{remark}
\label{remark8}For all vectors $x$ and $y$ in an inner
product space, we can refine \eqref{eq2.4}, by noting that for a positive integer $n\ge1$ we have
 
 \begin{align*}
  \left| {\left\langle {x,y} \right\rangle } \right|^{2n}&=\left( \left| {\left\langle {x,y} \right\rangle } \right|^2 \right)^n   
\\
 &\le  \left(\left(1-\beta\right)\left| {\left\langle {x,y} \right\rangle } \right|\left\| x \right\|\left\| y \right\| + \beta \left\| x \right\|^2 \left\| y \right\|^2    \right)^n \qquad\qquad\qquad \text{(by \eqref{eq2.4})}
\\
&= \sum\limits_{k = 0}^n {\left( \begin{array}{l}
	n \\ 
	k \\ 
	\end{array} \right)\left(\left(1-\beta\right)\left| {\left\langle {x,y} \right\rangle } \right|\left\| x \right\|\left\| y \right\|    \right)^k \left( \beta \left\| x \right\|^2 \left\| y \right\|^2    \right)^{n - k} }  \qquad \text{(by Binomial Theorem)}
\\
&= \sum\limits_{k = 0}^n {\left( \begin{array}{l}
	n \\ 
	k \\ 
	\end{array} \right) \left(1-\beta\right)^k \beta^{\left(n-k\right)} \left| {\left\langle {x,y} \right\rangle } \right|^k       \left\| x \right\|^{2n-k} \left\| y \right\|^{2n-k}    } 
\\
&\le \left(1-\beta\right)\left| {\left\langle {x,y} \right\rangle } \right|^n\left\| x \right\|^n\left\| y \right\|^n + \beta  \left\| x \right\|^{2n}\left\| y \right\|^{2n}
\\
&\le    \left\| x \right\|^{2n}\left\| y \right\|^{2n},
\end{align*}
the last two inequalities follow  from the convexity of $t^n$ $(n\ge1, t>0)$ and the classical Cauchy--Schwarz inequality \eqref{eq1.7}; respectively. Hence, we can deduce the sequence of inequalities
 \begin{align}
\left| {\left\langle {x,y} \right\rangle } \right|^{2n} 
&\le   \sum\limits_{k = 0}^n {\left( \begin{array}{l}
	n \\ 
	k \\ 
	\end{array} \right) \left(1-\beta\right)^k \beta^{\left(n-k\right)} \left| {\left\langle {x,y} \right\rangle } \right|^k       \left\| x \right\|^{2n-k} \left\| y \right\|^{2n-k}    } \label{eq3.11} 
\\
&\le \left(1-\beta\right)\left| {\left\langle {x,y} \right\rangle } \right|^n\left\| x \right\|^n\left\| y \right\|^n + \beta  \left\| x \right\|^{2n}\left\| y \right\|^{2n}
\nonumber\\
&\le    \left\| x \right\|^{2n}\left\| y \right\|^{2n},\nonumber
\end{align}
for all for positive integer $n\ge1$ and $\beta \in \left[0,1\right]$. Hence the first inequality in \eqref{eq3.11} gives the same result as \eqref{eq2.4} when $n=1$, however it refines \eqref{eq2.4} for all $n\ge2$.

In particular, for $n=1$ we refer to \eqref{eq2.4}, while for $n=2$ we get
 \begin{align*}
\left| {\left\langle {x,y} \right\rangle } \right|^{4} 
&\le    \beta^{2}        \left\| x \right\|^{4} \left\| y \right\|^{4} +2 \left(1-\beta\right)  \beta  \left| {\left\langle {x,y} \right\rangle } \right| \left\| x \right\|^{3} \left\| y \right\|^{3} + \left(1-\beta\right)^2   \left| {\left\langle {x,y} \right\rangle } \right|^2       \left\| x \right\|^{2} \left\| y \right\|^{2} 
\\
&\le \left(1-\beta\right)\left| {\left\langle {x,y} \right\rangle } \right|^2\left\| x \right\|^2\left\| y \right\|^2 + \beta  \left\| x \right\|^{4}\left\| y \right\|^{4}
\\
&\le    \left\| x \right\|^{4}\left\| y \right\|^{4} 
\end{align*}
for all $\beta \in \left[0,1\right]$. Hence, several numerical radius inequalities could be stated using \eqref{eq3.11} which may refine all stated results in this work. We leave the details to the interested reader. 
\end{remark}

\noindent {\bf Conclusion.}
It's well known that the inequality \eqref{eq1.3} refines \eqref{eq1.1}, while \eqref{eq1.2} is sharper than \eqref{eq1.3}. In \cite{KM}, it was proved that \eqref{eq1.14} is srtonger than \eqref{eq1.3} but it is not better than \eqref{eq1.2}. 

In this work, some numerical radius inequalities that refines all previous mentioned inequalities are established. More precisely, as we finished this work and among other presented inequalities, we proved that 
 \begin{align*}
\omega^{2p}\left(T\right) 
&\le \frac{1}{4}\beta \left\| \left| T \right|^{2p\alpha }+\left| {T^* } \right|^{2p\left( {1 - \alpha } \right)}\right\|^2 + \frac{1}{2}\left(1-\beta\right)\omega^{p}\left(T\right)\left\| \left| T \right|^{2p\alpha }+\left| {T^* } \right|^{2p\left( {1 - \alpha } \right)}\right\|
\\
&\le  \frac{1}{2}\beta \left\|\left| T \right|^{4p\alpha } +\left| {T^* } \right|^{4p\left( {1 - \alpha } \right)}\right\| + \frac{1}{2}\left( {1 - \beta } \right)\omega^{p}\left(T\right) \left\| \left| T \right|^{2p\alpha }+\left| {T^* } \right|^{2p\left( {1 - \alpha } \right)}\right\|
\\
&\le \frac{1}{2}  \left\|\left| T \right|^{4p\alpha } +\left| {T^* } \right|^{4p\left( {1 - \alpha } \right)}\right\| 
\end{align*}
for  any operator $T\in \mathscr{B}\left(\mathscr{H}\right)$, $p\ge1$, and $\alpha,\beta\in \left[0,1\right]$. 

Choosing $p=1$, $\alpha=\frac{1}{2}$ and $\beta = \frac{1}{3}$, in the first inequality above,  we get the inequality \eqref{eq3.9}. 
It is proved that the inequality \eqref{eq3.9} is much better than the strongest well known inequality \eqref{eq1.2}. As well as, the inequality \eqref{eq3.9}  refines  the inequalities \eqref{eq1.3}, \eqref{eq1.14} and \eqref{eq3.10}. 

As noted in Remark \ref{remark8}, the refinement of the Cauchy--Schwarz inequality  \eqref{eq3.11}  which generalizes  \eqref{eq2.4},   can be used to refine more numerical radius inequalities presented in this work.

\end{document}